\documentclass{amsart}
\usepackage{mathrsfs}
\usepackage{stmaryrd}

\newtheorem{theorem}{Theorem}[section]
\newtheorem{lemma}[theorem]{Lemma}
\newtheorem{corollary}[theorem]{Corollary}

\theoremstyle{definition}
\newtheorem{definition}[theorem]{Definition}
\newtheorem{example}[theorem]{Example}

\theoremstyle{remark}
\newtheorem{remark}[theorem]{Remark}

\numberwithin{equation}{section}

\def\FV{\mathrm{FV}}
\def\OFV{\mathrm{OFV}}
\def\d{\mathbf{d}}

\begin{document}

\title{Formal differential variables and an abstract chain rule}

\author{Samuel Alexander}


\date{}

\begin{abstract} One shortcoming of the chain rule is that it does not iterate: it gives the derivative of $f(g(x))$, but not (directly) the second or higher-order derivatives. We present iterated differentials and a version of the multivariable chain rule which iterates to any desired level of derivative. We first present this material informally, and later discuss how to make it rigorous (a discussion which touches on formal foundations of calculus). We also suggest a finite calculus chain rule (contrary to Graham, Knuth and Patashnik's claim that ``there's no corresponding chain rule of finite calculus'').
\end{abstract}

\maketitle

\section{Introduction}


Consider the following statement, uncontroversial
in an elementary calculus context ($*$): ``For all variables $u$ and $v$,
$d(uv)=v\,du+u\,dv$.''
In his popular calculus textbook \cite{stewart}, Stewart says:
\begin{quote}
    ...the \textbf{differential} $dx$ is an independent variable...
\end{quote}
So if $*$ really does hold for \emph{all} variables $u$ and $v$,
and if $x$ is a variable, and if (as Stewart says) $dx$ is also a variable,
then, by letting $u=x$ and $v=dx$, we get
$d(x\,dx) = dx\,dx + x\,ddx$.
We do not know whether Stewart intended us to make such an unfamiliar-looking
conclusion from his innocent-looking statement, but let's continue along
these lines and see where it leads us.
We will formalize this kind of computation using machinery from first-order logic,
and show that it leads to an elegant higher-order multivariable chain rule.

A weakness of the familiar chain rule is that it does not iterate: it tells us
how to find the \emph{first} derivative of $f(g(x))$, but it does not
tell us how to find second- or higher-order derivatives of the same (at least
not directly). Our abstract chain rule will iterate: the exact same rule which
tells us $d f(g(x))$ will also tell us $d^k f(g(x))$
for any integer $k>1$.

Our $d$ operator has some similarities with the $\Delta$ operator of
Huang et al \cite{huang2006chain}. Our work improves on theirs in that we
explicitly distinguish differential variables from others, so that the operator
we develop better reveals the connection to differentials. For example,
in Huang et al, one has $\Delta_1 e^{x_0}=e^{x_0}x_1$ and
$\Delta_2 e^{x_0}=e^{x_0}(x_1^2+x_2)$, which is equivalent to
our $de^{x_0}=e^{x_0}\,dx_0$ and
$d^2e^{x_0}=e^{x_0}(dx_0\,dx_0+ddx_0)$.
Besides better emphasizing the connection to differentials, the latter version should
also be more familiar, since we already routinely write things like
$de^x=e^x\,dx$ in elementary calculus classes.

\section{Computing iterated partial derivatives: informal examples}
\label{funsection}

In this section, we will informally describe a way to compute iterated partial derivatives of a
multivariable function. We will make the method formal in subsequent sections.

\begin{example}
    \label{ddxsquareexample}
    Compute the differential $dd\,x^2=d(dx^2)$, treating differential variables just like
    ordinary variables.
\end{example}

\begin{proof}[Solution]
    The differential $dx^2=2x\,dx$ involves two variables: $x$ and $dx$.
    Thus, $d(dx^2)$ will have two terms, one where we differentiate with respect to
    $x$ and multiply the result by $dx$, and one where we differentiate with respect
    to $dx$ and multiply the result by $ddx$:
    \begin{align*}
        ddx^2
            &= d(dx^2)\\
            &= d(2x\,dx)\\
            &= \frac{\partial (2x\,dx)}{\partial x}\,dx
                + \frac{\partial (2x\,dx)}{\partial dx}\,ddx\\
            &= 2\,dx\,dx + 2x\,ddx.
    \end{align*}
    Note that when we compute $\frac{\partial (2x\,dx)}{\partial x}$,
    we treat $dx$ as a variable independent from $x$, so
    $dx$ can be treated as a constant. Likewise when we compute
    $\frac{\partial (2x\,dx)}{\partial dx}$, $x$ is treated as a constant.
\end{proof}

\begin{example}
    \label{ddexexample}
    Compute the differential $dd\,e^x$, treating differential variables
    just like ordinary variables.
\end{example}

\begin{proof}[Solution]
    As in Example \ref{ddxsquareexample}, since $d\,e^x=e^x\,dx$,
    \begin{align*}
        dd\,e^x &= d\, (e^x\,dx)\\
            &= \frac{\partial (e^x\,dx)}{\partial x}\,dx
                + \frac{\partial (e^x\,dx)}{\partial dx}\,ddx\\
            &= e^x\,dx\,dx + e^x\,ddx.
    \end{align*}
\end{proof}

\begin{example}
    \label{ddfxexample}
    Compute $dd\,f(x)$, treating differential variables just like
    ordinary variables.
\end{example}

\begin{proof}[Solution]
    Just as above,
    \begin{align*}
        dd\,f(x) &=
            d(f'(x)\,dx)\\
            &= \frac{\partial (f'(x)\,dx)}{\partial x}\,dx
                + \frac{\partial (f'(x)\,dx)}{\partial dx}\,ddx\\
            &=
                f''(x)\,dx\,dx + f'(x)\,ddx.
    \end{align*}
\end{proof}

In a later section, we will formalize and prove a formal chain rule
(Corollary \ref{iteratedcoolcorollary}). For now, we will state it informally:

\begin{remark}
    (Abstract Chain Rule, stated informally)
    Let $T$ and $U$ be expressions and let $x$ be a non-differential variable.
    Assume $T$, $U$, and all of their sub-expressions are everywhere infinitely
    differentiable. Then
    \[
        d(T[x|U]) = (dT)[x|U],
    \]
    where the operator $[x|U]$ works by simultaneously replacing all occurrences
    of $x$ by $U$, all occurrences of $dx$ by $dU$, all occurrences of
    $d^2x$ by $d^2U$, and so on.
\end{remark}

The Abstract Chain Rule can be stated in English: ``substituting first and then
applying $d$ gives the same result as applying $d$ first and then substituting,
provided that when one substitutes $U$ for $x$, one also substitutes $dU$ for $dx$
and so on.''

\begin{example}
    \label{ex2example}
    Compute $(e^{x^2})''$.
\end{example}

\begin{proof}[Solution]
    By Example \ref{ddfxexample}, $(e^{x^2})''$ is the $dx\,dx$-coefficient of
    $dd\,e^{x^2}$. We compute:
    \begin{align*}
        dd\,e^{x^2}
            &= dd \, (e^x[x|x^2])\\
            &= (dd\,e^x)[x|x^2]
                &\mbox{(Abstract Chain Rule)}\\
            &= (e^x\,dx\,dx + e^x\,ddx)[x|x^2]
                &\mbox{(Example \ref{ddexexample})}\\
            &=
                e^{x^2}\,d(x^2)\,d(x^2) + e^{x^2}\,dd(x^2)
                    &\mbox{(Substituting)}\\
            &=
                e^{x^2}\,(2x\,dx)^2 + e^{x^2}\,(2\,dx\,dx + 2x\,ddx)
                    &\mbox{(Example \ref{ddxsquareexample})}\\
            &= (4x^2+2)e^{x^2}\,dx\,dx + 2xe^{x^2}\,ddx.
    \end{align*}
    The answer is the above $dx\,dx$-coefficient:
    \[
        (e^{x^2})'' = (4x^2+2)e^{x^2}.
    \]
\end{proof}

Our Abstract Chain Rule works for multivariable and higher-order derivatives, too.

\begin{example}
    The iterated total derivative
    \[
        d^3 \sin xy = d^3(\sin x \, [x|xy]) = (d^3 \sin x)[x|xy]
    \]
    encodes:
    \begin{itemize}
        \item
        $\partial^3 \sin xy/\partial x^3$ as its $dx\,dx\,dx$-coefficient.
        \item
        $\partial^3 \sin xy/\partial y^3$ as its $dy\,dy\,dy$-coefficient.
        \item
        $
            \frac{\partial^3 \sin xy}{\partial x\partial y\partial y}
            =
            \frac{\partial^3 \sin xy}{\partial y\partial x\partial y}
            =
            \frac{\partial^3 \sin xy}{\partial y\partial y\partial x}
        $
        times $3$ as its
        $dx\,dy\,dy=dy\,dx\,dy=dy\,dy\,dx$-coefficient (the fact that there
        are three ways to write this coefficient is why we write ``times 3'').
    \end{itemize}
\end{example}

In Sections \ref{formalizingtermssection}--\ref{abstractchainrulesecn} we will
formalize and prove the Abstract Chain Rule.
But first, we will connect these higher-order differentials to a more
concrete higher-order chain rule known as Fa\`a di Bruno's formula,
and also show how the same ideas lead to a finite calculus chain rule.

\section{Fa\`a di Bruno's formula}
\label{faadibrunosecn}

Fa\`a di Bruno's formula, named after the 19th century Italian priest
Francesco Fa\`a di Bruno, is a formula for the higher derivatives of $f(g(x))$.
See
\cite{johnson2002curious} and
\cite{craik2005prehistory}
for the history of Fa\`a di Bruno's formula
(see also \cite{osborne2017first} for related work in category theory
by another ACMS presenter).
The formula can be stated combinatorially:
\[
    f(g(x))^{(n)} = \sum_{\pi\in\Pi_n} f^{(|\pi|)}(g(x))\prod_{B\in\pi}g^{(|B|)}(x)
\]
where $\pi$ ranges over the set $\Pi_n$ of all partitions of $\{1,\ldots,n\}$
(so for each such partition $\pi$, $B$ ranges over the blocks in $\pi$).

The ideas of Section \ref{funsection} offer an intuitive way to understand the above
formula\footnote{Shortly after presenting this argument at ACMS, we realized
that the argument can actually be applied directly, without using iterated differentials
at all, yielding a shockingly short elementary proof of Fa\`a di Bruno's formula.
Examining the literature, we found that the basic idea is already
known \cite{ma2009higher} \cite{hardy2006combinatorics}, but both published proofs which
we found are actually proofs of more complicated multivariable generalizations of
Fa\`a di Bruno's formula. For the single-variable special case, the idea
(essentially the same idea which we presented using iterated differentials at ACMS)
is so simple that it can be written with a single sentence \cite{onesentence}.}. For any
partition $\pi=\{B_1,\ldots,B_k\}$ of $\{1,\ldots,n\}$,
let $I(\pi)$ be the expression
\[
    I(\pi)=f^{(k)}(x) \,d^{|B_1|}x \,d^{|B_2|}x \,\cdots \,d^{|B_k|}x
\]
involving iterated differentials as in Section \ref{funsection}.
By an inductive argument, one can check that
\[
    d^{n}f(x) = \sum_{\pi\in\Pi_n} I(\pi)
\]
(for the inductive step, consider the different ways of obtaining a
partition $\pi'\in\Pi_{n+1}$
from a partition $\pi\in\Pi_n$: one can either add $\{n+1\}$ as a new block,
which corresponds to changing $f^{(k)}(x)$ to $f^{(k+1)}(x)dx$ when using
the product rule to calculate $dI(\pi)$; or one can add $n+1$ to existing block $B_i$
of $\pi$, which corresponds to changing $d^{|B_i|}x$ to $d^{|B_i|+1}x$
when using the product rule to calculate $dI(\pi)$).

By similar reasoning as in Examples \ref{ddfxexample} and \ref{ex2example},
$f(g(x))^{(n)}$ is the $(dx)^n$-coefficient of
$d^n f(g(x))=d^n (f(x)[x|g(x)])=(d^n f(x))[x|g(x)]$.
Thus $f(g(x))^{(n)}$ is the $(dx)^n$-coefficient of
\[
    \sum_{\pi\in\Pi_n} I(\pi)[x|g(x)]
    = \sum_{\pi\in\Pi_n} f^{(|\pi|)}(g(x))\prod_{B\in\pi} d^{|B|}g(x).
\]
One can check that $d^{|B|}g(x)=g^{(|B|)}(x)d^{|B|}x+o$ where $o$
is a sum of terms involving higher-order differentials (which can be
ignored because they contribute nothing to the $(dx)^n$-coefficient
we seek). Fa\`a di Bruno's formula follows.


\section{Application to finite calculus}

The ideas in this paper also lead to a chain rule for the so-called finite calculus.
The finite calculus is described in Section 2.6 of Graham, Knuth and Patashnik's
\emph{Concrete Mathematics} \cite{concrete}. In finite calculus, one defines an
operator $\Delta$ on functions by $\Delta f(x) = f(x+1)-f(x)$. This operator has
many surprising analogies with differentiation, but Graham et al
claim: ``there's no corresponding chain rule of finite calculus, 
because there's no nice form for $\Delta f(g(x))$.''
To the contrary, since $\Delta x = (x+1)-x=1$,
an equivalent way to write $\Delta f(x)$ is
\[
    \Delta f(x) = f(x+\Delta x) - f(x).
\]
One can then easily check that
\[
    \Delta f(g(x))
    = f(g(x+\Delta x))-f(g(x))
    = f(g(x) + \Delta g(x)) - f(g(x)),
\]
which can be expressed as a chain rule
\[
    \Delta (f(x)\llbracket x|g(x)\rrbracket)=(\Delta f(x))\llbracket x|g(x)\rrbracket,
\]
where $\llbracket x|g(x)\rrbracket$ operates by replacing $x$
by $g(x)$ and $\Delta x$ by $\Delta g(x)$.

Of course, to make this rigorous, it would be necessary to work in a formal
language so as to carefully track which ``$1$''s are ``$\Delta x$''s.
For example, if $f(x)=1/(1+x^2)$, we want $f(x)\llbracket x|g(x)\rrbracket$
to be $1/(1+g(x)^2)$, not $\Delta g(x)/(\Delta g(x)+g(x)^2)$, even though
$\Delta x=1$. We will not go through the necessary formalism in this paper,
but it would be very similar to the formalism required for the $d$ operator,
which we devote the whole rest of the paper to.

\section{Formalizing terms}
\label{formalizingtermssection}

In this section, we will formalize the terms (or expressions) of differential calculus.
We attempt to make this formalization self-contained. The machinery we develop here is
very similar to the machinery used to define terms in first-order logic, except
that we assume more structure on the set of variables than is assumed in first-order
logic.

Note that one could strongly argue that elementary calculus already implicitly operates
on terms, abusing language to call terms ``functions''. For example, $x\mapsto x^2$
and $y\mapsto y^2$ are two names for the exact same function. Yet, nevertheless, in
elementary calculus, the expressions $x^2$ and $y^2$ are \emph{not} interchangeable
\cite{hamkins}. Evidently, such discrepancies point to the fact that elementary calculus
really is done using formal terms, implicitly. In the following,
we make it explicit.

\begin{definition}
    (Variables) 
    We fix a set of \emph{variables} defined inductively as follows.
    \begin{enumerate}
        \item
        For the base step, we fix a countably infinite set $\{x_0,x_1,\ldots\}$ of
        distinct elements called
        \emph{precalculus variables}, and we declare them to be variables.
        \item
        Inductively, for every variable $v$, we fix a new variable $dv$, which we
        call a \emph{differential variable}; we do this in such a way as to
        satisfy the following requirement (we write $d^n v$
        for $ddd\cdots dv$ where $d$ occurs $n$ times):
        \begin{itemize}
            \item
            (Unique Readability)
            For all $n,m\in\mathbb N$, for all variables $v$ and $w$, if
            $d^nv$ is the same variable as $d^mw$, then $n=m$ and $v=w$.
        \end{itemize}
    \end{enumerate}
    We write $\mathscr V$ for the set of variables.
\end{definition}

Examples of variables include $x_1$, $x_{50}$, $dx_0$, $ddx_3$, $d^4 x_{50}$
(shorthand for $ddddx_{50}$), and so on.
The unique readability property guarantees that, for example,
$dx_1$ is not the same variable as
$dx_2$ or $ddx_3$ or $dddx_1$, etc. We allow $n$ or $m$ to be $0$ in the unique
readability requirement, so, for example, $x_1$ and $dx_1$ are not the same variable
(since $d^0x_1$ denotes $x_1$). Every variable
is either a precalculus variable (in which case it is $x_n$ for some $n\in\mathbb N$)
or a differential
variable (in which case it is $d^m x_n$ for some $n,m\in\mathbb N$ with $m>0$).

\begin{definition}
    (Constant symbols and function symbols)
    \begin{enumerate}
        \item
        We fix a distinct set $\{\overline r\}_{r\in\mathbb R}$ of \emph{constant symbols}
        for the real numbers. For any $r\in\mathbb R$, $\overline r$ is the \emph{constant
        symbol for $r$}.
        \item
        For every $n\in\mathbb N$ with $n>0$, we fix a distinct set $\{\overline f\}_f$
        of \emph{$n$-ary function symbols}, where $f$ ranges over the set of all
        functions from $\mathbb R^n$ to $\mathbb R$.
        For any such $f$, $\overline f$ is the \emph{$n$-ary function symbol for $f$}.
    \end{enumerate}
    We make these choices in such a way that no variable is a constant symbol, no
    variable is an $n$-ary function symbol (for any $n$), and no constant symbol is
    an $n$-ary function symbol (for any $n$).
\end{definition}

For example, the exponential function $\exp$ gives rise to a $1$-ary
(or \emph{unary}) function symbol $\overline{\exp}$.
The addition function $+$ gives rise to a $2$-ary (or \emph{binary})
function symbol $\overline{+}$.

\begin{definition}
    \label{termdefn}
    (Terms)
    We define the \emph{terms of differential calculus} (or simply \emph{terms})
    inductively as follows.
    \begin{enumerate}
        \item
        Every variable $v$ is a term.
        \item
        Every constant symbol is a term.
        \item
        For all $n\in\mathbb N$ ($n>0$),
        for every $f:\mathbb R^n\to\mathbb R$,
        for all terms $U_1,\ldots,U_n$,
        $\overline f(U_1,\ldots,U_n)$ is a term.
    \end{enumerate}
\end{definition}

Examples of terms include $\overline 5$, $\overline \pi$,
$x_1$, $dx_2$, $\overline{\sin}(x_1)$,
$\overline{+}(x_0,x_1)$, and so on.
We often abuse notation and suppress the overlines and possibly parentheses when
writing terms. For example, we might write $\sin x_0$ instead of $\overline{\sin}(x_0)$,
$\cos \pi$ instead of $\overline{\cos}(\overline\pi)$, and so on.
For certain well-known functions, we sometimes abuse notation further,
for example, writing:
\begin{itemize}
    \item $x_0+x_1$ instead of $\overline{+}(x_0,x_1)$;
    \item $2x_0$ instead of $\overline{\cdot}(\overline 2,x_0)$;
    \item $x_0\,dx_1$ instead of $\overline{\cdot}(x_0,dx_1)$;
    \item $x^2_0$ instead of $\overline{x\mapsto x^2}(x_0)$;
    \item $e^{x_1}$ instead of $\overline{exp}(x_1)$;
    \item $x_0\,dx_1 + x_1\,dx_0$ instead of
        $\overline{+}(\overline{\cdot}(x_0,dx_1),\overline{\cdot}(x_1,dx_0))$;
    \item and so on.
\end{itemize}
This should cause no confusion in practice.

\begin{definition}
    \label{terminterpdefn}
    (Term interpretation)
    \begin{itemize}
        \item
        By an \emph{assignment}, we mean a function $s:\mathscr V\to\mathbb R$
        (recall that $\mathscr V$ is the set of variables).
        \item
        Let $s$ be an assignment.
        For every term $T$, we define the \emph{interpretation} $T^s\in\mathbb R$
        of $T$ (according to $s$) by induction on term complexity as follows.
        \begin{enumerate}
            \item
            If $T$ is a constant symbol $\overline r$, then $T^s=r$.
            \item
            If $T$ is a variable $v$, then $T^s=s(v)$.
            \item
            If $T$ is $\overline f(U_1,\ldots,U_n)$ for some
            $f:\mathbb R^n\to\mathbb R$ and terms $U_1,\ldots,U_n$,
            then $T^s=f(U^s_1,\ldots,U^s_n)$.
        \end{enumerate}
    \end{itemize}
\end{definition}

For example, if $s(x_0)=5$, then $\overline{\exp}(x_0)^s=e^5$.
If $s(x_0)=9$ and $s(dx_0)=0.1$, then $(x\,dx)^s=9\cdot 0.1=0.9$.

\begin{definition}
    (Free variables)
    We define the \emph{free variables} $\FV(T)$ of a term $T$ as follows.
    \begin{enumerate}
        \item
        If $T$ is a constant symbol, then $\FV(T)=\emptyset$ (the empty set).
        \item
        If $T$ is a variable $v$, then $\FV(T)=\{v\}$.
        \item
        If $T$ is $\overline f(U_1,\ldots,U_n)$ for some $f:\mathbb R^n\to\mathbb R$
        and terms $U_1,\ldots,U_n$, then
        \[
            \FV(T) = \FV(U_1)\cup \cdots \cup \FV(U_n).
        \]
    \end{enumerate}
\end{definition}

For example, $\FV(\overline 5)=\emptyset$,
$\FV(x_6)=\{x_6\}$, $\FV(dx_2)=\{dx_2\}$ (note that $x_2$ is \emph{not}
a free variable of $dx_2$), $\FV(e^{x_0+x_1})=\{x_0,x_1\}$,
$\FV(x_1\,dx_2)=\{x_1,dx_2\}$.

\begin{lemma}
    \label{independenceofsvlemma}
    Suppose $T$ is a term, $v$ is a variable, and $s$ is an assignment.
    If $v\not\in \FV(T)$, then $T^s$
    does not depend on $s(v)$.
\end{lemma}

\begin{proof}
    By induction.
\end{proof}

\begin{definition}
    \label{semanticequivdefn}
    (Semantic equivalence)
    If $T$ and $U$ are terms, we declare $T\equiv U$ (and say that $T$ and $U$
    are \emph{semantically equivalent}) if for every assignment $s$,
    $T^s=U^s$.
\end{definition}

For example, $\sin(x_0+2\pi)\equiv \sin x_0$,
by which we mean
$\overline \sin(\overline +(x_0,\overline{2\pi})))\equiv\overline \sin(x_0)$.

\subsection{Formal derivatives}

\begin{definition}
    \label{orderedfreevarsdefn}
    (Ordered free variables)
    If $T$ is a term, we define the \emph{ordered free variables} $\OFV(T)$ to be
    the finite sequence whose elements are the free variables $\FV(T)$ of $T$ (each
    appearing exactly one time in the sequence), ordered such that:
    \begin{itemize}
        \item
        Whenever $0<n<m$ then $d^n x_i$ precedes $d^m x_j$.
        \item
        Whenever $0<i<j$ then $d^n x_i$ precedes $d^n x_j$.
    \end{itemize}
\end{definition}

For example,
\[
    \OFV(e^{x_1+x_3+x_2+x_2+x_{99}}\,dx_1\,d^3x_1\,dx_2\,d^{50}x_0)
    =
    (x_1, x_2, x_3, x_{99}, dx_1, dx_2, d^3x_1, d^{50}x_0).
\]

\begin{definition}
    \label{assignmentshiftdefn}
    If $s$ is an assignment, $w$ is a variable, and $r\in\mathbb R$,
    we write $s(w|r)$ for the assignment defined by
    \[
        s(w|r)(v) = \begin{cases}
            r & \mbox{if $v$ is $w$}\\
            s(v) & \mbox{otherwise}.
        \end{cases}
    \]
\end{definition}

In other words, $s(w|r)$ is the assignment which is identical to $s$ except that
it overrides $s$'s output on $w$, mapping $w$ to $r$ instead.

\begin{lemma}
    \label{assignmentequalsitselflemma}
    For any assignment $s$ and variable $v$, $s(v|s(v))=s$.
\end{lemma}

\begin{proof}
    Trivial.
\end{proof}

\begin{definition}
    (Everywhere-differentiability)
    Let $T$ be a term, $w$ a variable.
    We say that \emph{$T$ is everywhere-differentiable with respect to $w$}
    if for every assignment $s$, the limit
    \[
        \lim_{h\to 0}\frac{T^{s(w|s(w)+h)}-T^s}{h}
    \]
    converges to a finite real number.
\end{definition}

\begin{lemma}
    \label{welldefinednesslemma}
    Let $T$ be a term with $\OFV(T)=(v_1,\ldots,v_n)\not=\emptyset$,
    and let $w$ be a variable.
    Assume $T$ is everywhere-differentiable with respect to $w$.
    For all $r_1,\ldots,r_n$, let
    \[
        f(r_1,\ldots,r_n) = \lim_{h\to 0}\frac{T^{s(w|s(w)+h)}-T^s}{h}
    \]
    where $s$ is some assignment such that each $s(v_i)=r_i$.
    Then $f:\mathbb R^n\to\mathbb R$ is well-defined.
\end{lemma}

\begin{proof}
    In other words, for any $r_1,\ldots,r_n\in\mathbb R$,
    $f(r_1,\ldots,r_n)$ does not depend on the choice of $s$, as long
    as each $s(v_i)=r_i$. This follows from Lemma \ref{independenceofsvlemma}
    since $T$ has no free variables other than $v_1,\ldots,v_n$.
\end{proof}

\begin{definition}
    \label{termderivativedefn}
    If $T$ is a term
    with $\OFV(T)=(v_1,\ldots,v_n)$, $w$ is a variable, and $T$ is everywhere-differentiable
    with respect to $w$, then we define the \emph{derivative of $T$ with respect
    to $w$}, a term, written $\frac{\partial T}{\partial w}$, as
    \[
        \frac{\partial T}{\partial w} = \overline f(v_1,\ldots,v_n)
    \]
    where $f$ is as in Lemma \ref{welldefinednesslemma}.
    We define $\frac{\partial T}{\partial w}$ to be the term $\overline 0$ if
    $\FV(T)=\emptyset$.
\end{definition}

\begin{example}
    \label{termderivexamples}
    (Some example term derivatives)
    \begin{enumerate}
        \item
        $\partial x_0/\partial x_0\equiv 1$.
        \item
        $\partial x_0/\partial x_1\equiv 0$.
        \item
        $\partial x_0/\partial dx_0\equiv 0$.
        \item
        $\partial (e^{x_1x_2}\,dx_1)/\partial x_1
            \equiv x_2e^{x_1x_2}\,dx_1$.
    \end{enumerate}
\end{example}

\begin{proof}
(1) The function $f$ of Lemma \ref{welldefinednesslemma}
is
\[
    f(r)=\lim_{h\to 0}\frac{x_0^{s(x_0|s(x_0)+h)}-x_0^s}{h}
\]
(for any assignment $s$ with $s(x_0)=r$).
By Definitions \ref{terminterpdefn} and \ref{assignmentshiftdefn}
this simplifies to
$f(r)=\lim_{h\to 0}\frac{s(x_0)+h-s(x_0)}{h}=1$.
The claim follows.

(2) The function $f$ of Lemma \ref{welldefinednesslemma}
is
\[
    f(r)=\lim_{h\to 0}\frac{x_0^{s(x_1|s(x_1)+h)}-x_0^s}{h}
\]
(where $s(x_0)=r$).
This simplifies to
$f(r)=\lim_{h\to 0}\frac{s(x_0)-s(x_0)}{h}=0$.
The claim follows.

(3) Similar to (2).

(4) By Definition \ref{orderedfreevarsdefn},
$\OFV(e^{x_1x_2}\,dx_1) = (x_1,x_2,dx_1)$.
So, letting $v_1=x_1$, $v_2=x_2$, $v_3=dx_1$, the function $f$
of Definition \ref{welldefinednesslemma} is
\[
    f(r_1,r_2,r_3) = \lim_{h\to 0}
    \frac{
        (e^{x_1x_2}\,dx_1)^{s(v_1|s(v_1)+h)}
        -
        (e^{x_1x_2}\,dx_1)^s
    }h
\]
(where each $s(v_i)=r_i$).
By Definitions \ref{terminterpdefn} and \ref{assignmentshiftdefn}
this simplifies to
\[
    f(r_1,r_2,r_3)
    =\lim_{h\to 0}\frac{
        e^{(r_1+h)r_2}r_3 - e^{r_1r_2}r_3
    }h,
\]
which is $r_2e^{r_1r_2}r_3$ by calculus. The claim follows.
\end{proof}

Another way to prove Example \ref{termderivexamples} would be to use the following lemma.

\begin{lemma}
    \label{subtletechnicallemma}
    For each term $T$, variable $w$, and assignment $t$,
    if $T$ is everywhere-differentiable with respect to $w$, then
    \[
        \left(\frac{\partial T}{\partial w}\right)^t
        =
        \lim_{h\to 0}\frac{ T^{t(w|t(w)+h)} - T^t }h.
    \]
\end{lemma}

\begin{proof}
    If $\FV(T)=\emptyset$, the lemma is trivial. Assume not.
    Let $(v_1,\ldots,v_n)=\OFV(T)$.
    By definition,
    $\frac{\partial T}{\partial w}=\overline f(v_1,\ldots,v_n)$,
    where $f:\mathbb R^n\to\mathbb R$ is such that for all
    $r_1,\ldots,r_n\in\mathbb R$, for any assignment $s$ with each $s(v_i)=r_i$,
    \[
        f(r_1,\ldots,r_n) = \lim_{h\to 0}
        \frac{T^{s(w|s(w)+h)}-T^s}h.
    \]

    In particular, let each $r_i=t(v_i)$. Then:
    \begin{align*}
        \left(\mbox{$\frac{\partial T}{\partial w}$}\right)^t
        &=
        \overline f(v_1,\ldots,v_n)^t
            &\mbox{(Definition \ref{termderivativedefn})}\\
        &=
        f(t(v_1),\ldots,t(v_n))
            &\mbox{(Definition \ref{terminterpdefn})}\\
        &=
        f(r_1,\ldots,r_n)
            &\mbox{(Choice of $r_1,\ldots,r_n$)}\\
        &=
        \lim_{h\to 0}\frac{ T^{t(w|t(w)+h)}-T^t }h,
            &\mbox{(Since each $t(v_i)=r_i$)}\\
    \end{align*}
    as desired.
\end{proof}

\begin{definition}
    (Term total differentials)
    Suppose $T$ is a term.
    We say $T$ is \emph{everywhere totally differentiable}
    if $T$ is everywhere-differentiable with respect to every variable.
    If so, we define the \emph{total differential} $\d T$, a term,
    as follows.
    If $\FV(T)=\emptyset$ then we define $\d T = \overline 0$.
    Otherwise, let $\OFV(T)=(v_1,\ldots,v_n)$ and define
    \[
        \d T
        =
        \frac{\partial T}{\partial v_1}
        dv_1 + \cdots + \frac{\partial T}{\partial v_n} dv_n.
    \]
    Furthermore, we inductively define $\d^1 T$ to be $\d T$ and,
    whenever $\d^n T$ is defined and is everywhere totally differentiable,
    we define $\d^{n+1} T=\d \d^n T$.
\end{definition}

For example,
\begin{align*}
\d (x_1\,dx_2) &= \frac{\partial (x_1\,dx_2)}{\partial x_1}dx_1
    + \frac{\partial (x_1\,dx_2)}{\partial dx_2}ddx_2\\
    &\equiv
    dx_1\,dx_2 + x_1\,ddx_2.
\end{align*}

\begin{lemma}
    \label{variablesubsetlemma}
    If term $T$ is everywhere totally differentiable and
    if $v_1,\ldots,v_n$ are distinct variables
    such that $\FV(T)\subseteq\{v_1,\ldots,v_n\}$, then
    \[
        \d T \equiv
        \frac{\partial T}{\partial v_1}dv_1 + \cdots + \frac{\partial T}{\partial v_n}dv_n.
    \]
\end{lemma}

\begin{proof}
    Follows from the commutativity of addition and the fact that
    clearly $\frac{\partial T}{\partial v_i}\equiv \overline 0$ if $v_i\not\in\FV(T)$.
\end{proof}

In order to prove an abstract chain rule in Section \ref{abstractchainrulesecn},
we will need a form of the classical multivariable chain rule, expressed for formal
terms. For this purpose, we first introduce shorthand for finite summation
notation\footnote{It is also possible to incorporate summation notation
formally into Definition \ref{termdefn}, but the details are complicated.
See \cite{alexander2013first}.}.

\begin{definition}
    If $m>0$ is an integer and $T_1,\ldots,T_m$ are terms,
    we write $\sum_{i=1}^m T_i$ (or just $\sum_i T_i$ if
    no confusion results) as shorthand for $T_1+\cdots+T_m$.
\end{definition}

\begin{lemma}
    \label{technicallemma}
    (Classic Multivariable Chain Rule for Terms)
    Suppose $f:\mathbb R^n\to\mathbb R$.
    Suppose $\vec T=(T_1,\ldots,T_n)$ are terms with each
    $\FV(T_i)\subseteq \{v_1,\ldots,v_m\}$ (where $v_1,\ldots,v_m$ are distinct).
    Assume that $\overline f(\vec T)$ and $T_1,\ldots,T_n$ are
    everywhere totally differentiable.
    Then for all $1\leq i\leq m$,
    \[
        \frac{\partial (\overline f(\vec T))}{\partial v_i}
        \equiv
        \sum_{j=1}^n \overline{f_j}(\vec T)\frac{\partial T_j}{\partial v_i},
    \]
    where $f_j=D_j f$ (the partial derivative of $f$ (in the usual sense) with
    respect to its $j$th argument).
\end{lemma}

\begin{proof}
    Let $s$ be an assignment
    and fix $1\leq i\leq m$. We must show (Definition \ref{semanticequivdefn})
    that
    \[
        \left(
            \frac{\partial (\overline f(\vec T))}{\partial v_i}
        \right)^s
        =
        \left(
            \sum_{j=1}^n \overline{f_j}(\vec T)\frac{\partial T_j}{\partial v_i}
        \right)^s.
    \]
    Define functions $F, G_{j}:\mathbb R\to\mathbb R$
    ($1\leq j\leq n$)
    by
    \begin{align*}
        F(z) &= \overline{f}(\vec T)^{s(v_i|z)},\\
        G_{j}(z) &= T^{s(v_i|z)}_j.
    \end{align*}
    For all $1\leq j\leq n$ and $z\in\mathbb R$,
    \begin{align*}
        F(z) &= \overline{f}(\vec T)^{s(v_i|z)}
                &\mbox{(Definition of $F_i$)}\\
            &= f(T^{s(v_i|z)}_1,\ldots,T^{s(v_i|z)}_n)
                &\mbox{(Definition \ref{terminterpdefn})}\\
            &= f(G_{1}(z),\ldots,G_{n}(z)),
                &\mbox{(Definition of $G_{j}$)}\\
        \mbox{so ($*$) }F'(z)
            &= \mbox{$\sum_j$} f_j(G_{1}(z),\ldots,G_{n}(z))G'_{j}(z)
                &\mbox{(Classic multivar.\ chain rule)}
    \end{align*}
    (the hypotheses of the classic multivariable chain rule are implied by
    the everywhere-total-differentiability of $\overline f(\vec T)$ and each $T_i$,
    by Lemma \ref{subtletechnicallemma}).
    So armed, we compute:
    \begin{align*}
        \left(
            \mbox{$\frac{\partial (\overline f(\vec T))}{\partial v_i}$}
        \right)^s
        &=
            \lim_{h\to 0}\frac{\overline f(\vec T)^{s(v_i|s(v_i)+h)}-\overline f(\vec T)^s}h
            &\mbox{(Lemma \ref{subtletechnicallemma})}\\
        &=
            \lim_{h\to 0}\frac{F(s(v_i)+h)-F(s(v_i))}h
            &\mbox{(Def.\ of $F$)}\\
        &=
            F'(s(v_i))
            &\mbox{(Def.\ of $F'$)}\\
        &= \mbox{$\sum_j$}
            f_j(G_{1}(s(v_i)),\ldots,G_{n}(s(v_i)))G'_{j}(s(v_i))
            &\mbox{(By ($*$))}\\
        &= \mbox{$\sum_j$}
            f_j(T^{s(v_i|s(v_i))}_1,\ldots,T^{s(v_i|s(v_i))}_n)G'_{j}(s(v_i))
            &\mbox{(Def.\ of $G_{j}$)}\\
        &= \mbox{$\sum_j$} f_j(T^s_1,\ldots,T^s_n)G'_{j}(s(v_i))
            &\mbox{(Lemma \ref{assignmentequalsitselflemma})}\\
        &= \mbox{$\sum_j$} f_j(T^s_1,\ldots,T^s_n)
            \lim_{h\to 0}\frac{G_{j}(s(v_i)+h)-G_{j}(s(v_i))}h
            &\mbox{(Def.\ of $G'_{j}$)}\\
        &= \mbox{$\sum_j$} f_j(T^s_1,\ldots,T^s_n)
            \lim_{h\to 0}\frac{T^{s(v_i|s(v_i)+h)}-T^{s(v_i|s(v_i))}}h
            &\mbox{(Def.\ of $G_{j}$)}\\
        &= \mbox{$\sum_j$} f_j(T^s_1,\ldots,T^s_n)
            \lim_{h\to 0}\frac{T^{s(v_i|s(v_i)+h)}-T^{s}}h
            &\mbox{(Lemma \ref{assignmentequalsitselflemma})}\\
        &= \mbox{$\sum_j$} f_j(T^s_1,\ldots,T^s_n)
            \left(\mbox{$\frac{\partial T_j}{\partial v_i}$}\right)^s
            &\mbox{(Lemma \ref{subtletechnicallemma})}\\
        &=
            \left(
                \mbox{$\sum_{j=1}^n$}
                    \overline{f_j}(\vec T)
                    \mbox{$\frac{\partial T_j}{\partial v_i}$}
            \right)^s,
            &\mbox{(Def.\ \ref{terminterpdefn})}
    \end{align*}
    as desired.
\end{proof}

Note that in Lemma \ref{technicallemma} the assumption that $\overline f(\vec T)$
is everywhere totally differentiable does not automatically imply that
$T_1,\ldots,T_n$ are everywhere totally differentiable.
For example, $f$ could be the function $f(x,y)=x$
in which case $f(T_1,T_2)$ would be everywhere totally differentiable iff $T_1$
is everywhere totally differentiable, regardless of the behavior of $T_2$.

\section{An Abstract Chain Rule}
\label{abstractchainrulesecn}

Recall that $\mathscr V$ denotes the set of all variables. Let $\mathscr T$ denote the
set of all terms.

\begin{definition}
    \label{extensiondefn}
    For any $\phi_0:\mathscr V\to\mathscr T$, the \emph{extension of $\phi_0$ to all terms}
    is the function $\phi:\mathscr T\to\mathscr T$ defined by induction as follows:
    \begin{enumerate}
        \item If $T$ is a constant symbol then $\phi(T)=T$.
        \item If $T$ is a variable then $\phi(T)=\phi_0(T)$.
        \item If $T$ is $\overline f(S_1,\ldots,S_n)$ then
            $\phi(T)=\overline f(\phi(S_1),\ldots,\phi(S_n))$.
    \end{enumerate}
\end{definition}

\begin{lemma}
    \label{substitutionlemma}
    Let $\phi_0:\mathscr V\to\mathscr T$ and
    let $\phi$ be the extension of $\phi_0$ to all terms. Then:
    \begin{enumerate}
        \item
        (The Substitution Lemma)
        For any assignment $s$, if $\phi(s)$ is the assignment defined by
        $\phi(s)(v)=\phi(v)^s$, then for every term $T$,
        $\phi(T)^s=T^{\phi(s)}$.
        \item
        For all terms $T$ and $U$, if $T\equiv U$ then $\phi(T)\equiv\phi(U)$.
    \end{enumerate}
\end{lemma}

\begin{proof} (1) By induction on $T$. If $T$ is a constant symbol or variable, the
    claim is trivial. Otherwise, $T$ is $\overline f(U_1,\ldots,U_n)$.
    Then
    \begin{align*}
        \phi(T)^s
        &= \overline f(\phi(U_1),\ldots,\phi(U_n))^s
            &\mbox{(Definition \ref{extensiondefn})}\\
        &= f(\phi(U_1)^s,\ldots,\phi(U_n)^s)
            &\mbox{(Definition \ref{terminterpdefn})}\\
        &= f(U_1^{\phi(s)},\ldots,U_n^{\phi(s)})
            &\mbox{(Induction)}\\
        &= T^{\phi(s)}.
            &\mbox{(Definition \ref{terminterpdefn})}
    \end{align*}
    \item (2) Assume $T\equiv U$. For any assignment $s$,
    if $\phi(s)$ is as in (1), then
    $T^{\phi(s)}=U^{\phi(s)}$ by Definition \ref{semanticequivdefn}.
    Thus $\phi(T)^s=\phi(U)^s$ by (1).
    By arbitrariness of $s$, $\phi(T)\equiv\phi(U)$.
\end{proof}

\begin{definition}
    Say $\phi_0:\mathscr V\to\mathscr T$ \emph{respects $d$}
    if for each variable $v$, $\phi_0(d v)\equiv\d \phi_0(v)$.
\end{definition}

\begin{definition}
    (Strong differentiability)
    \begin{enumerate}
        \item
        We define the \emph{subterms} of a term $T$ by induction as follows.
        If $T$ is a variable or constant symbol, then $T$ is its own lone
        subterm. If $T$ is $\overline f(U_1,\ldots,U_n)$, then the subterms
        of $T$ are $T$ itself along with the subterms of each $U_i$.
        \item
        A term $T$ is \emph{strongly differentiable} if every subterm of
        $T$ is everywhere totally differentiable.
    \end{enumerate}
\end{definition}

Thus, a term is strongly differentiable if it is built up from pieces which are
everywhere totally differentiable. An example of a term which is everywhere totally
differentiable but not strongly differentiable is $|x_0|^2$, which is
everywhere totally differentiable despite having a subterm $|x_0|$ which is not.
Note that the ordinary chain rule for $f(g(x))'$ fails when $f(x)=x^2$ and $g(x)=|x|$
(these functions fail the chain rule's hypotheses):
$(|x|^2)'=2x$, but $|x|'$ is undefined at $x=0$.
We avoid such traps in the following theorem by requiring strong differentiability.

\begin{theorem}
    \label{mainthm}
    (General Abstract Chain Rule)
    Let $\phi_0:\mathscr V\to\mathscr T$ and assume that $\phi_0(v)$
    is strongly differentiable for every variable $v$.
    Let $\phi$ be the extension of $\phi_0$ to all terms.
    If $T$ is strongly differentiable and $\phi_0$ respects $d$, then
    $\d \phi(T) \equiv \phi(\d T)$.
\end{theorem}

\begin{proof}
    By induction on $T$.
    If $T$ is a constant symbol, the theorem is trivial.
    If $T$ is a variable, the theorem reduces to the statement that $\phi_0$
    respects $d$, which is one of the hypotheses.
    It remains to consider the case when $T$ is
    $\overline f(\vec T)$ where $f:\mathbb R^m\to\mathbb R$
    and $\vec T=T_1,\ldots,T_m$ are
    simpler terms.
    Then $T_1,\ldots,T_m$ are subterms of $T$, so, since $T$ is strongly
    differentiable, it follows that $T_1,\ldots,T_m$ are strongly differentiable.
    By induction, each $\d \phi(T_i) \equiv \phi(\d T_i)$.
    Let $\{v_1,\ldots,v_\ell\}=\FV(\phi(T_1))\cup \cdots \cup \FV(\phi(T_m))$.
    For the rest of the proof, whenever $S$ is a term and $v$ is a variable, we will write
    $S_v$ for $\frac{\partial S}{\partial v}$.
    Let $\overrightarrow{\phi(T)}$
    denote
    $\phi(T_1),\ldots,\phi(T_m)$.
    We calculate:
    \begin{align*}
        {} & \d \phi(\overline f(\vec T))\\
        &\equiv \mbox{$\sum_{i=1}^\ell$} \phi(\overline f(\vec T))_{v_i}dv_i
            &\mbox{(Lemma \ref{variablesubsetlemma})}\\
        &= \mbox{$\sum_i$} \overline f(\overrightarrow{\phi(T)})_{v_i}dv_i
            &\mbox{(Definition \ref{extensiondefn})}\\
        &\equiv
            \mbox{$\sum_i$} \mbox{$\sum_{j=1}^m$} \overline{f_j}(\overrightarrow{\phi(T)})
                \phi(T_j)_{v_i}dv_i
                &\mbox{(Lemma \ref{technicallemma})}\\
        &\equiv
            \mbox{$\sum_j$} \overline{f_j}(\overrightarrow{\phi(T)})
            \mbox{$\sum_i$} \phi(T_j)_{v_i} dv_i
                &\mbox{(Basic algebra)}\\
        &\equiv
            \mbox{$\sum_j$} \overline{f_j}(\overrightarrow{\phi(T)})
            \d \phi(T_j)
                &\mbox{(Lemma \ref{variablesubsetlemma})}\\
        &\equiv
            \mbox{$\sum_j$} \overline{f_j}(\overrightarrow{\phi(T)})
            \phi(\d T_j)
                &\mbox{(Induction Hypothesis)}\\
        &=
            \phi\left(
                \mbox{$\sum_j$}
                \overline{f_j}(\vec T)
                \d T_j
            \right)
                &\mbox{(Definition \ref{extensiondefn})}\\
        &\equiv
            \phi\left(
                \mbox{$\sum_j$} \overline{f_j}(\vec T)
                \mbox{$\sum_{i=1}^\ell$} (T_j)_{v_i}\,dv_i
            \right)
                &\mbox{(Lemma \ref{variablesubsetlemma})}\\
        &\equiv
            \phi\left(
                \mbox{$\sum_i$} \mbox{$\sum_j$} \overline{f_j}(\vec T)(T_j)_{v_i}\,dv_i
            \right)
                &\mbox{(Basic algebra)}\\
        &\equiv
            \phi(\mbox{$\sum_i$}
                \overline{f}(\vec T)_{v_i}
            dv_i)
                &\mbox{(Lemma \ref{technicallemma})}\\
        &\equiv
            \phi( \d \overline{f}(\vec T))
                &\mbox{(Lemma \ref{variablesubsetlemma})}
    \end{align*}
    (in the last few lines, we use Lemma \ref{substitutionlemma} part 2).
\end{proof}

A weakness of the familiar chain rule is that it does not iterate.
The following corollary shows that the abstract chain rule does iterate.

\begin{corollary}
    For all $\phi_0$, $\phi$ and $T$ as in Theorem \ref{mainthm},
    for all $k\in\mathbb N$ ($k>0$),
    if $\d^{\ell} T$ exists and is strongly differentiable for all $\ell<k$, then
    \[
        \d^k \phi(T) \equiv \phi(\d^k T).
    \]
\end{corollary}

\begin{proof}
    By repeated applications of Theorem \ref{mainthm}.
\end{proof}

In Sections \ref{funsection} and \ref{faadibrunosecn} we used a
special case of Theorem \ref{mainthm}
which we will now formalize. Recall that a precalculus variable is one that
is not of the form $dv$ for any variable $v$.

\begin{definition}
    \label{substitutionrespectingdifferentialsdefn}
    (Variable substitution respecting differentials)
    Let $v$ be a precalculus variable, $U$ a term
    such that $\d^k U$ is strongly differentiable for all $k$.
    For every term $T$, we will define the \emph{result of substituting
    $U$ for $v$ in $T$ while respecting differentials}, written $T[v|U]$,
    as follows.
    First, we define $\phi_0:\mathscr V\to\mathscr T$ so that:
    \begin{enumerate}
        \item
        $\phi_0(v) = U$.
        \item
        For every $k>0$, $\phi_0(d^k v) = \d^k U$.
        \item
        For all variables $w$ not of either of the above two forms, $\phi_0(w)=w$.
    \end{enumerate}
    We define $T[v|U]$ to be $\phi(T)$ where $\phi$ is the extension of $\phi_0$
    to all terms (Definition \ref{extensiondefn}).
\end{definition}

\begin{corollary}
    \label{coolcorollary}
    (Abstract Chain Rule)
    Let $U,v$ be as in Definition \ref{substitutionrespectingdifferentialsdefn}.
    If term $T$ is strongly differentiable, then
    \[
        \d (T[v|U]) \equiv (\d T)[v|U].
    \]
\end{corollary}

\begin{proof}
    If $\phi_0$ is as in Definition \ref{substitutionrespectingdifferentialsdefn}
    then evidently $\phi_0$ satisfies the hypotheses of
    Theorem \ref{mainthm}.
    The corollary then immediately follows from Theorem \ref{mainthm}.
\end{proof}

\begin{corollary}
    \label{iteratedcoolcorollary}
    (Iterated Abstract Chain Rule)
    Let $v, T, U$ be as in Corollary \ref{coolcorollary}.
    For all $k>0$, if $\d^\ell T$ is strongly differentiable for all $\ell<k$,
    then
    \[
        \d^k (T[v|U]) \equiv (\d^k T)[v|U].
    \]
\end{corollary}

\begin{proof}
    By repeated applications of Corollary \ref{coolcorollary}.
\end{proof}

\section*{Acknowledgments}

We gratefully acknowledge Bryan Dawson, Tevian Dray, and the reviewers for
generous comments and feedback.

\bibliographystyle{plain}
\bibliography{faa}

\end{document}